\documentclass[graybox]{svmult}

\usepackage{type1cm}        % activate if the above 3 fonts are
\usepackage{multicol}        % used for the two-column index
\usepackage[bottom]{footmisc}% places footnotes at page bottom

\usepackage{newtxtext}       %
\usepackage{newtxmath}       % selects Times Roman as basic font

\makeindex             % used for the subject index
\usepackage{enumerate,amsmath,bbm}
\usepackage{hyperref}
\newcommand{\set}[1]{\left\{#1\right\}}
\newcommand{\norm}[1]{\left\lVert#1\right\rVert}

\newcommand{\bprod}[1]{\bigl(#1\bigr)}
\newcommand{\ind}[1]{\mathbbm{1}_{#1}}

\DeclareMathOperator{\AC}{\!\textit{AC\/}}
\DeclareMathOperator{\hyperIFI}{\sideset{_1}{_1}{\mathop F}}

\newcommand{\R}{\mathbb{R}}

\newcommand{\pr}{\mathrm{P}}
\newcommand{\ex}[1]{\mathrm{E}\left[\,#1\,\right]}

\DeclareMathOperator{\indicatorFun}{\mathbbm{1}}

\begin{document}
%\nocite{MISHURA2001421}
\title*{Gaussian processes with Volterra kernels}
%\titlerunning{Gaussian processes with Volterra kernels}
\author{Yuliya Mishura, Georgiy Shevchenko and Sergiy Shklyar}
\institute{Yuliya Mishura \at
	Department of Probability Theory, Statistics and Actuarial Mathematics,
	Taras Shevchenko National University of Kyiv,
	64, Volodymyrs'ka St.,
	01601 Kyiv, Ukraine\\
	\email{myus@univ.kiev.ua}
	\and  Georgiy Shevchenko\at
	Department of Probability Theory, Statistics and Actuarial Mathematics,
	Taras Shevchenko National University of Kyiv,
	64, Volodymyrs'ka St.,
	01601 Kyiv, Ukraine\\
	\email{zhora@univ.kiev.ua}
	\and Sergiy Shklyar \at
	Department of Probability Theory, Statistics and Actuarial Mathematics,
	Taras Shevchenko National University of Kyiv,
	64, Volodymyrs'ka St.,
	01601 Kyiv, Ukraine\\
	\email{shklyar@univ.kiev.ua}
}

\maketitle

\abstract{
	We study   Volterra processes $X_t = \int_0^t K(t,s)dW_s$, where $W$ is a standard Wiener process, and the kernel has the form $K(t,s) = a(s) \int_s^t b(u) c(u-s) du$. This form generalizes the Volterra kernel for fractional Brownian motion (fBm) with Hurst index $H>1/2$.    We establish smoothness properties of $X$, including continuity and H\"{o}lder property. It happens that its H\"{o}lder smoothness is close to well-known H\"{o}lder smoothness of fBm but is a bit worse. We give a comparison with fBm for any smoothness theorem.  Then we  investigate the problem of inverse representation of $W$ via $X$ in the case  where $c\in L^1[0,T]$ creates a Sonine pair, i.e.\@ there exists $h\in L^1[0,T]$ such that $c * h \equiv 1$.
It is a natural extension of the respective property of fBm that generates the same filtration with the underlying Wiener process. Since the inverse representation of the Gaussian processes under consideration are based on the properties of  Sonine pairs,  we provide several examples of Sonine pairs, both well-known and new.
}

\section*{Introduction} Among various classes of Gaussian processes, consider the class of the processes admitting the integral representation via some Wiener process. Such processes   arise in finance, see  e.g. \cite{Boguslavskaya2018}.  They are the natural extension of fractional Brownian motion (fBm) which admits the integral representation via the  Wiener process, and the Volterra kernel of its representation consists of power functions. The solution of many problems related to fBm is based on the  H\"{o}lder properties of its trajectories. Therefore it is interesting to consider the smoothness properties of Gaussian processes admitting the integral representation via some Wiener process, with the representation kernel  that generalizes the kernel in the representation of fBm. The next question is what properties should the kernel have in order for the Wiener process and the corresponding Gaussian process to generate the same filtration. It turned out that the functions in the kernel should form, in a specific way, so called Sonine pair, property that the components of the kernel   generating fBm have. Thus, the properties of the Gaussian process turned out to be directly related to the analytical properties of the generating kernel. The  present work is devoted to the study of these properties. It is organized as follows. Section \ref{sec-2} is devoted to the smoothness properties of the Gaussian processes generated by Volterra kernels. Assumptions which supply the existence and continuity of the Gaussian process are provided. Then the H\"{o}lder properties are established. They have certain features. Namely, under reasonable assumptions on the kernel we can establish only H\"{o}lder property up to order $1/2$ while fBm with Hurst index $H$ has H\"{o}lder property of the trajectories up to order $H$, and for $H>1/2$ (exactly the case from which we start) fBm has  better smoothness properties. In this connection, we establish the conditions of smoothness that is comparable with the one for fBm, but only on any interval separated from zero. Finally, we establish the conditions on the kernel supplying  H\"{o}lder property at zero.  Section \ref{sec-3} describes how the generalized fractional calculus related to a Volterra process with Sonine kernel can be used to invert the corresponding covariance operator.  Section \ref{sec:examples} contains examples of Sonine pairs, and Section \ref{app} contains all necessary auxiliary results.

\section{Gaussian Volterra processes and their smoothness properties}\label{sec-2}

Let $(\Omega, \mathcal{F}, \mathbf{F}=\{\mathcal{F}_t, t\ge 0\}, \mathbf{P})$ be a stochastic basis with filtration, and let $W=\{W_t, t\ge 0\}$ be a Wiener process adapted to this filtration.
Consider a Gaussian process of the form
\begin{equation}\label{eq:volterra}
X_t = \int_0^t K(t,s) dW_s
\end{equation}
where $K\in L^2([0,T]^2)$ is a Volterra kernel, i.e. $K(t,s) = 0$ for $s>t$. Obviously, $X$ is also adapted to the filtration $\mathbf{F}$. Recall that a very common example of such process is a fractional Brownian motion (fBm) with Hurst index $H$, i.e., a Gaussian process $B^H=\{B^H_t, t\ge 0\}$, admitting a representation $$B^H_t=\int_0^t K(t,s) \, dW_s,$$
with some Wiener process $W$ and Volterra kernel \begin{equation}\label{eq:formzall}\begin{aligned}
K(t,s) &= c_H
s^{1/2 - H}
\Big(
(t  (t-s))^{H-1/2} -  (H- 1/2)
\int_s^t u^{H-3/2} (u-s)^{H-1/2}\, du \Big)\indicatorFun_{0<s<t},
\end{aligned}\end{equation}
where
 $c_H = \left(
\frac{2 H \,\Gamma(\frac32 - H)}{\Gamma(H+\frac12)\,\Gamma(2-2H)}
\right)^{1/2}.
$
If $H>\frac12$, then the kernel $K$ from \eqref{eq:formzall} can be simplified
to
\begin{equation}\label{equ-bigH}
K(t,s) =
\left(H-\frac12\right)
c_H
s^{1/2 - H}
\int_s^t u^{H-1/2} (u-s)^{H-3/2}\, du.
\end{equation}

Now, motivated by a fractional Brownian motion with $H>1/2$, we assume that the kernel in the  representation   \eqref{eq:volterra} is given by
\begin{equation}\label{eq:abckernel}
K(t,s) = a(s) \int_s^t b(u) \, c(u-s) du,
\end{equation}
where $a,b,c:[0,T]\to \R$ are some measurable functions. Since many applications of fBm are based on its smoothness properties, we consider what properties of functions $a,b,c$ provide a certain smoothness of the process $X$  which, in the case under consideration, takes the form

\begin{equation}\label{eq:volterra1}
X_t = \int_0^t \left(a(s) \int_s^t b(u) \, c(u-s) \, du \right)dW_s, \; t\in [0,T].
\end{equation}

Our first goal is to  investigate the assumptions which supply the existence and continuity of process $X$. Considering   $L$-spaces, we put, as is standard, $1/\infty=0$ and $1/0=\infty.$

\begin{theorem}\label{thm:lem_XEcontin}
Assume that
{\settowidth{\leftmargini}{(K2)~\,}
\begin{enumerate}
\renewcommand{\theenumi}{{\upshape(K\arabic{enumi})}}
\renewcommand{\labelenumi}{\upshape(K\arabic{enumi})}
	\item\label{cond:K1}
	  $a\in L^{p}[0,T]$,
	$b\in L^{q}[0,T]$, and
	$c\in L^{r}[0,T]$
	for   $p\in[2,\infty]$,
	$q\in[1,\infty]$,
	$r\in[1,\infty]$,
	such that
	$1/p + 1/q + 1/r \le \frac32$.
\end{enumerate}}\relax
Then
\begin{equation*}%\label{eq:k-l2estimate}
\sup_{t\in[0,T]} \norm{K(t,\,\cdot\,)}_{L^2[0,t]}<\infty,
\end{equation*}
which means that the process $X$ is well defined.

If, in addition, $1/p + 1/r < \frac32$,
then the process $X$ has a continuous modification.
\end{theorem}
\begin{remark}\label{rem1}  In the case of fBm    with $H>1/2$
we have $a(t) = \left(H-\frac12\right) c_H t^{1/2-H}$,\linebreak
$b(t) = t^{H-1/2}$ and $c(t) = t^{H-3/2}$. Therefore, $p$ can be any number such that\linebreak $\frac{1}{2}\mathbin{>}\frac{1}{p}\mathbin{>}H{-}\frac{1}{2}$, $q$ can be  any number from $[1,\infty]$, and $r$ can be any number such that $1 > \frac{1}{r}> \frac{3}{2}-H$. It means that  both  conditions of Theorem \ref{thm:lem_XEcontin} are satisfied if we put $\frac{1}{p}=H-\frac{1}{2}+\frac{\varepsilon}{3}$, $\frac{1}{q}=\frac{\varepsilon}{3}$ and $\frac{1}{r}= \frac{3}{2}-H+\frac{\varepsilon}{3}$, where
$0 \mathbin{<} \epsilon \mathbin{<} \min\bigl(3\bigl(H-\frac12\bigr), \: 3(1-H), \: \frac12\bigr)$.
\end{remark}
\begin{proof}
For both statements, without loss of generality,
we can assume that $1/q + 1/r \ge 1$. Considering
   statement 2) we can assume that
$q<\infty$.

1)\quad Extend the functions $a$, $b$, $c$ onto the entire
set $\mathbb{R}$ assuming $a(s) = b(s) =\linebreak c(s) = 0$
for all $s\not\in[0,T]$.
Extend the kernel $K(t,s)$ assuming
$K(t,s) = 0$ for $s\not\in[0,t]$.\linebreak
Then we have
\begin{equation}\label{eq:Kts=conv}
K(t,s) = a(s) \, (b \indicatorFun_{[0,t]} * \tilde c)(s),
\quad \mbox{for all} \quad
0\le t \le T,
\quad
s \in \mathbb{R},
\end{equation}
where $\tilde c(v) = c(-v)$.
By Young's convolution inequality \eqref{neq:Young}
\begin{equation}\label{eq:tempbi0t}
\|b \indicatorFun_{[0,t]} * \tilde c\|_{(1/q + 1/r - 1)^{-1}}
\le
\|b \indicatorFun_{[0,t]}\|_q \,
\|\tilde c\|_r \le
\|b\|_q \, \|c\|_r.
\end{equation}
(Here we applied inequality $1/q + 1/r \ge 1$.)
By H\"older inequality \eqref{neq:Holder}
for non-conjugate exponents
\begin{align}
\|K(t,\,\cdot\,)\|_{ (1/p + 1/q + 1/r - 1)^{-1}}
&=
\|a \, (b \indicatorFun_{[0,t]} * \tilde c)\|_{ (1/p+1/q+1/r-1)^{-1}} \nonumber \\
&\le
\|a\|_p \,
\|b \indicatorFun_{[0,t]} * \tilde c\|_{(1/q+1/r-1)^{-1}}
\le
\|a\|_p \,
\|b\|_q \,
\|c\|_r .
\label{neq:tempKtd}
\end{align}
Hence
$K(t,\,\cdot\,) \in L^{ (1/p + 1/q + 1/r - 1)^{-1}}[0,t]$.
Since
$(1/p + 1/q + 1/r - 1)^{-1} > 2$,
we conclude that
$K(t,\,\cdot\,) \in    L^2[0,t]$, and it follows from \eqref{neq:tempKtd} that the norms are uniformly bounded.
It completes the proof of the first
statement.

2)\quad Let $0\le t_1 < t_2 \le T$. It follows from  \eqref{eq:Kts=conv} that
\begin{equation}
K(t_2,s) - K(t_1,s) =
a(s)\, (b\indicatorFun_{(t_1,t_2]} * \tilde c) (s),
\qquad s\in\mathbb{R}.
\label{eq:difKt2sKt1s}
\end{equation}
Similarly to \eqref{eq:tempbi0t} and \eqref{neq:tempKtd},
\begin{gather*}
\|b \indicatorFun_{(t_1,t_2]} * \tilde c\|_{(1/q+1/r-1)^{-1}}
\le
\|b \indicatorFun{(t_1,t_2]}\|_q \,
\|\tilde c\|_r \le
\|b \indicatorFun_{(t_1,t_2]}\|_q \, \|c\|_r,
\end{gather*}
and
\begin{multline*}
  \|K(t_2,\,\cdot\,) - K(t_1,\,\cdot\,)\|_{(1/p + 1/q +  1/r - 1)^{-1}}
 =
\|a \, (b \indicatorFun_{(t_1,t_2]} * \tilde c)\|_{(1/p + 1/q + 1/r - 1)^{-1}} \\
 \le
\|a\|_p \,
\|b \indicatorFun_{(t_1,t_2]} * \tilde c\|_{(1/q+1/r-1)^{-1}}
 \le
\|a\|_p \,
\|b \indicatorFun_{(t_1,t_2]}\|_q \,
\|c\|_r.
 \end{multline*}
Notice that
 $2 < (1/p + 1/q + 1/r - 1)^{-1}$,
and the function
$K(t_2,\,\cdot\,) - K(t_1,\,\cdot\,)$
is zero-valued outside the interval $[0,t_2]$.
Apply the inequality \eqref{neq:LL} between
the norms in $L^2[0,t_2]$ and
$L^{(1/p + 1/q +1/r - 1)^{-1}}[0,t_2]$:
\begin{align*}
\|K(t_2,\,\cdot\,) - K(t_1,\,\cdot\,)\|_2
 &\le
\|K(t_2,\,\cdot\,) - K(t_1,\,\cdot\,)\|_{(1/p + 1/q + 1/r - 1)^{-1}}
t_2^{\frac32 - 1/p - 1/q - 1/r}
\\ &\le
\|a\|_p \,
\|b \indicatorFun_{(t_1,t_2]}\|_q \,
\|c\|_r
t_2^{\frac32 - 1/p - 1/q - 1/r}
\le C \|b \indicatorFun_{(t_1,t_2]}\|_q,
\end{align*}
with $C = T^{\frac32 - 1/p - 1/q - 1/r}
\|a\|_p \, \|c\|_r$.
Hence
\begin{align*}
\ex{ (X_{t_2} - X_{t_1})^2} &=
\|K(t_2,\,\cdot\,) - K(t_1,\,\cdot\,)\|_2^2
\le
C^2 \|b \indicatorFun_{(t_1,t_2]}\|_q^2
\\ &=
C^2 \left( \int_{t_1}^{t_2} |b(s)|^q ds\right)^{\!2/q}
= (F(t_2) - F(t_1))^{2/q},
\end{align*}
where
\[
F(t) = C^q \int_0^{t} |b(s)|^q\, ds
\]
is a nondecreasing function.
By Lemma~\ref{lem:GContF},
the process $\{X_t,\; t\in[0,T]\}$ has a continuous modification.
\end{proof}

Now, let us establish the conditions supplying H\"{o}lder properties of $X$.
\begin{lemma}\label{lem:pabc_Holder}
Assume that
$a \in L^p [0,T]$,
$b \in L^q [0,T]$, and
$c \in L^r [0,T]$
with
$ p\in[2,\infty]$,
$ q\in(1,\infty]$,
$ r\in[1,\infty]$,
so that
$1/p + 1/r \ge \frac12$ and
$1/p + 1/q + 1/r < \frac32$.
Then the stochastic process
$X$ defined by \eqref{eq:volterra1}
has a modification satisfying H\"older condition
up to order
$\frac32 - 1/p - 1/q - 1/r$.
\end{lemma}
\begin{remark} As it was mentioned in Remark \ref{rem1}, in the case of fractional Brownian motion, for any small positive $\varepsilon$, we have chose $p$, $q$ and $r$ so that $1\le 1/p + 1/q+1/r \le 1+\varepsilon$. Therefore in conditions of  Lemma \ref{lem:pabc_Holder} we get for fBm  H\"{o}lder property only up to order $1/2$ while in reality we know H\"{o}lder property   up to order $H>1/2$.
\end{remark}
\begin{proof}
Extend the functions $a$, $b$, $c$ and $K(t,s)$
as it was done in the proof of Theorem~\ref{thm:lem_XEcontin}.
Let $0 \le t_1 < t_2 \le T$.
We are going to find an upper bound for $\|K(t_2,\,\cdot\,) -  K(t_1,\,\cdot\,)\|_2$
using a representation \eqref{eq:difKt2sKt1s}.

By H\"older inequality for non-conjugate exponents \eqref{neq:Holder},
\[
\| b \indicatorFun_{(t_1,t_2]} \|_{ (\frac32 - 1/p - 1/r)^{-1}} \le
\|b\|_q \, \|\indicatorFun_{(t_1,t_2]}\|_{ (\frac32 - 1/p - 1/q - 1/r)^{-1}}
=
\|b\|_q (t_2 - t_1)^{\frac32 - 1/p - 1/q - 1/r} .
\]
Here we use that
$1/p + 1/q + 1/r \le \frac32$.
By Young's convolution inequality \eqref{neq:Young},
\begin{align*}
\|b \indicatorFun_{[t_1,t_2]} * \tilde c\|_{ (\frac12 - 1/p)^{-1}} &\le
\| b \indicatorFun_{[t_1,t_2]} \|_{(\frac32 - 1/p - 1/r)^{-1}}
\| \tilde c \|_r \\
&\le
\|b\|_q \|c\|_r
(t_2 - t_1)^{\frac32 - 1/p - 1/q - 1/r} .
\end{align*}
Here $\tilde c(v) = c(-v)$;
we used inequalities
$r\ge 1$,
$\frac12 \le 1/p + 1/r < \frac32$
so $(\frac32 - 1/p - 1/r)^{-1} \ge 1$, and
$p \ge 2$, so
 $(\frac12 - 1/p)^{-1} \ge 2$.

Again, by H\"older inequality for non-conjugate exponents,
\begin{align}
\|K(t_2,\,\cdot\,) - K(t_1,\,\cdot\,)\|_2
&=
\|a\,(b \indicatorFun_{(t_1,t_2]} * \tilde c)\|_2
\le
\|a\|_p\,
\|b \indicatorFun_{(t_1,t_2]} * \tilde c\|_{1/(\frac12 - 1/p)}
\nonumber \\ & \le
\|a\|_p \, \|b\|_q \, \|c\|_r \,
(t_2 - t_1)^{\frac32 - 1/p - 1/q - 1/r} .
\label{neq:Kt2-Kt1}
\end{align}

Hence
\begin{align*}
\ex{ (X_{t_2} - X_{t_1})^2} &=
\|K(t_2,\,\cdot\,) - K(t_1,\,\cdot\,)\|_2^2
\\ &\le
\|a\|^2_p\, \|b\|^2_q\, \|c\|^2_r\,
(t_2 - t_1)^{3 - 2(1/p + 1/q + 1/r)}.
\end{align*}
By Corollary~\ref{cor:Gausscontp},
the process $\{X_t,\; t\in[0,T]\}$
has a modification that satisfies
H\"older condition up to order  $\frac32 - 1/p - 1/q - 1/r$.
\end{proof}
The following statement  follows, to some extent, from
Lemma~\ref{lem:pabc_Holder}. Now we drop the condition
$1/p + 1/r \ge \frac12$, and simultaneously  relax the
assertion of the mentioned lemma.

\begin{theorem}\label{thm:cor:cpabc_Holder}
Let
$a \in L^{p} [0,T]$,
$b \in L^{q} [0,T]$, and
$c \in L^{r} [0,T]$
with
$p \in [2, \infty]$,
$q \in (1, \infty]$, and
$r \in [1, \infty]$,
which satisfy the inequality
$1/p + 1/q + 1/r < \frac32$.
Then the stochastic process
$X$ defined in \eqref{eq:volterra1}
has a modification that satisfies H\"older condition
up to order
$\frac32 - 1/q - \max(\frac12, \: 1/p + 1/r)$.
\end{theorem}
\begin{remark}
For the fBm with Hurst index $H\in\bigl(\frac12, 1\bigr)$
and functions $a$, $b$ and $c$ and exponents
$p$, $q$ and $r$ defined in Remark~\ref{rem1},
Theorem~\ref{thm:cor:cpabc_Holder} provides
H\"older condition up to order
$\frac32 - \frac{\epsilon}{3} -
\max\bigl(\frac12, \: 1 +  \frac{2\epsilon}{3}
\bigr) = \frac12 - \epsilon$.
However, since conditions of Lemma~\ref{lem:pabc_Holder} holds true in this case,
Lemma~\ref{lem:pabc_Holder} gives the same result.
\end{remark}
\begin{proof}
Let $r' = \left(\max(1/r,\: \frac12 - 1/p)\right)^{-1}$.
Then $r' \in [1, + \infty]$,
$r' \le r$, $c \in L^{r'}[0,T]$,
$1/p + 1/r' \ge \frac12$,
$1/p + 1/q + 1/r' < \frac32$.
Applying Lemma~\ref{lem:pabc_Holder} to the  functions $a$, $b$, $c$
and exponents $p$, $q$ and $r'$,
we obtain   that the process $X$ has a modification that satisfies
H\"older condition up to order
$\frac32 - 1/p - 1/q - 1/r' = \frac32 - 1/q - \max(\frac12, \: 1/p + 1/r)$.
 \end{proof}

Now,  let us formulate stronger  conditions on the functions $a$, $b$ and $c$, supplying better H\"{o}lder properties on any interval, ``close'' to $[0,T]$, but not on the whole $[0,T]$.

\begin{theorem}\label{thm:lem:HolderBi}
Let $t_1 \ge 0$, $t_2 \ge 0$ and $t_1 + t_2 < T$.
Let the functions $a$, $b$ and $c$
and constants $p$, $p_1$, $q$, $q_1$,
$r$, and $r_1$  satisfy the following assumptions
\begin{align*}
a&\in L^p[0,T]\cap L^{p_1}[t_1,T],\; where\;
 2 \le p \le p_1;
\\
b&\in L^q[0,T]\cap L^{q_1}[t_1+t_2,T],\; where\;
 1 < q \le q_1;
\\
c&\in L^r[0,T]\cap  L^{r_1}[t_2,T],\; where\;
 1 \le r \le r_1.
 \end{align*}
Also, let $1/p+1/q+1/r \le \frac32$, and
$1/q_1 +  \max\left(\frac12, \: 1/p + 1/r_1, \:  1/p_1 + 1/r\right) < \frac32$.

Then the stochastic process
$\{X_t,\; t\in[t_1+t_2,\:T]\}$
has a modification that satisfies H\"older condition
up to order
$\frac32 - 1/q_1 - \max\!\left(\frac12, \: 1/p+1/r_1, \: 1/p_1 + 1/r\right)$.
\end{theorem}
\begin{remark}
Consider the fBm
with Hurst index $H \in \bigl(\frac12, 1\bigr)$ on interval $[0, T]$.
Define the functions $a$, $b$ and $c$ and exponents $p$, $q$ and $r$ as
it is done in Remark~\ref{rem1}.
Let $p_1 = q_1 = r_1 = 3/\epsilon$,
where $\epsilon$ comes from Remark~\ref{rem1},
and let $t_1 = t_2 = t_0/2$ for some
$t_0 \in (0, T)$.
Then the conditions of Theorem~\ref{thm:lem:HolderBi}
are satisfied, and, according to Theorem~\ref{thm:lem:HolderBi}
the fBm has a modification which
satisfies H\"older condition
in the interval $[t_0,\: T]$ up to order
$\frac32 - \frac{\epsilon}{3}
- \max\bigl(\frac12, \:\allowbreak
\frac32 - H + \frac{2\epsilon}{3}, \:\allowbreak
H - \frac12 + \frac{2\epsilon}{3}\bigr)
= H - \epsilon$.
This is equivalent to the fact that the fBm
satisfies H\"older condition in the interval
$[t_0, T]$ up to order $H$.
\end{remark}

\begin{proof}
Let us extend  the function $a(s)$, $b(s)$, $c(s)$ and $K(t,s)$
as it was done in the proof of Theorem~\ref{thm:lem_XEcontin}.
With this extension, \eqref{eq:abckernel} holds true
for all $t\in[0,T]$ and $s\in\mathbb{R}$.
Denote
\begin{align*}
a_1(s) &= a(s) \indicatorFun_{[0,t_1)}, &  b_1(s) &= b(s) \indicatorFun_{[t_1+t_2,\:T]}, \\
a_2(s) &= a(s) \indicatorFun_{[t_1,T]}, &  c_1(s) &= c(s) \indicatorFun_{[t_2,\:T]},
\\
\tilde c(s) &= c(-s), &
\tilde c_1(s) &= c_1(-s) = c(-s) \indicatorFun_{[-T,\:{-}t_2]}(s) .
\end{align*}

The process $\{X_t, \; t\in[0,T]\}$ is well-defined
according to Theorem~\ref{thm:lem_XEcontin}.
We consider the increments of the process
$\{X_t, \; t\in[t_1 + t_2,\:T]\}$.
Let $t_3$ and $t_4$ be such that $t_1+t_2 \le t_3 < t_4 < T$.
Then
\begin{align*}
K(t_4,s) - K(t_3,s) &= a(s) \int_{t_3}^{t_4} b(u) c(u-s) \, du
= a(s) \int_{t_3}^{t_4} b_1(u) c(u-s) \, du\\
&\hspace{6.2cm} \mbox{for all $s\in\mathbb{R}$}; \\
K(t_4,s) - K(t_3,s)
&= a_1(s) \int_{t_3}^{t_4} b_1(u) c_1(u-s) \, du
\quad \mbox{for} \quad 0\le s < t_1; \\
K(t_4,s) - K(t_3,s) &= a_2(s) \int_{t_3}^{t_4} b_1(u) c(u-s) \, du
\quad \mbox{for} \quad t_1 \le s \le  T.
\end{align*}
Thus, for all $s\in\mathbb{R}$
\begin{align*}
K(t_4,s) - K(t_3,s)
&= a_1(s) \int_{t_3}^{t_4} b_1(u) c_1(u-s) \, du
+ a_2(s) \int_{t_3}^{t_4} b_1(u) c(u-s) \, du
\\ &=
a_1(s) \, (b_1 \indicatorFun_{(t_3,t_4]} * \tilde c_1) (s)
       + a_2(s) \, (b_1 \indicatorFun_{(t_3,t_4]} * \tilde c) (s).
\end{align*}
Functions $a_1$, $b_1$ and $c_1$
with exponents $p$, $q_1$ and $\left(\max(1/r_1,\:\frac12 - 1/p)\right)^{-1}$
satisfy conditions of Lemma~\ref{lem:pabc_Holder}.
Functions $a_2$, $b_1$ and $c$
with exponents $p_1$, $q_1$ and $\left(\max(1/r,\:\frac12 - 1/p_1)\right)^{-1}$
also satisfy conditions of Lemma~\ref{lem:pabc_Holder}.
By inequality \eqref{neq:Kt2-Kt1} in the proof of Lemma~\ref{lem:pabc_Holder},
\begin{gather*}
\|a_1 \, (b_1 \indicatorFun_{(t_3,t_4]} * \tilde c_1)\|_2 \le
\|a_1\|_p \, \|b_1\|_{q_1} \,
\|c_1\|_{1/{\max(1/r_1,\:\frac12 - 1/p)}} \,
(t_4 - t_3)^{\lambda_1}, \\
\|a_2 \, (b_1 \indicatorFun_{(t_3,t_4]} * \tilde c)\|_2 \le
\|a_2\|_{p_1} \, \|b_1\|_{q_1} \,
\|c\|_{1/{\max(1/r,\:\frac12 - 1/p_1)}} \,
(t_4 - t_3)^{\lambda_2}
\end{gather*}
where
\begin{gather*}
\lambda_1 = \frac32 - \frac{1}{p} - \frac{1}{q_1} - \max\!\left(\frac{1}{r_1},\:\frac12 - \frac{1}{p}\right)
= \frac32 - \frac{1}{q_1} - \max\!\left(\frac12, \: \frac{1}{p} + \frac{1}{r_1}\right), \\
\lambda_2 = \frac32 - \frac{1}{p_1} - \frac{1}{q_1} - \max\!\left(\frac{1}{r},\:\frac12 - \frac{1}{p_1}\right)
= \frac32- \frac{1}{q_1} - \max\!\left(\frac12, \: \frac{1}{p_1} + \frac{1}{r}\right).
\end{gather*}
Denote
\[
\lambda = \min(\lambda_1, \lambda_2) =
\frac32- \frac{1}{q_1} - \max\!\left(\frac12, \: \frac{1}{p} + \frac{1}{r_1}, \: \frac{1}{p_1} + \frac{1}{r}\right).
\]
Then
\[
\|K(t_4,\,\cdot\,) - K(t_3,\,\cdot\,)\|_2 \le
\|a_1 \, (b_1 \indicatorFun_{(t_3,t_4]} * \tilde c_1)\|_2 +
\|a_2 \, (b_1 \indicatorFun_{(t_3,t_4]} * \tilde c)\|_2
\le C \, (t_4 - t_3)^\lambda,
\]
where
\begin{align*}
C &=
\|a_1\|_p \, \|b_1\|_{q_1} \,
\|c_1\|_{1/{\max(r_1,\:\frac12 - 1/p)}} \,
T^{\lambda_1-\lambda} \\ &\quad+
\|a_2\|_{p_1} \, \|b_1\|_{q_1} \,
\|c\|_{1/{\max(r,\:\frac12 - 1/p_1)}} \,
T^{\lambda_2-\lambda}.
\end{align*}
Finally,
\begin{align*}
\ex{\bprod{X_{t_4}-X_{t_3}}^2} &\le
\int_{t_3}^{t_4} (K(t_4,s) - K(t_3,s))^2 \, ds
\\ &=
\|K(t_4,\,\cdot\,) - K(t_3,\,\cdot\,)\|_2^2 \le
C^2 (t_4-t_3)^{2\lambda} .
\end{align*}
By Corollary~\ref{cor:Gausscontp}, the stochastic process
$\{X_t,\; t\in[t_1+t_2,\:T]\}$
has a modification that satisfies H\"older condition
up to order  $\lambda$.
\end{proof}

The next result, namely, Lemma  \ref{lem:cor:c2pabc_Holder},  generalizes Lemma~\ref{lem:pabc_Holder}
and Theorem~\ref{thm:cor:cpabc_Holder}.  It allows us to apply  the mentioned lemma directly
 to the power functions   $a(s) = s^{-1/p_0}$ and
 $c(s) = s^{-1/r_0}$.

\begin{lemma}\label{lem:cor:c2pabc_Holder}
 Let $p_0 \in (0,{+}\infty]$,
$q_0 \in (1,{+}\infty]$, $r_0 \in (0,{+}\infty]$ with
$1/p_0 + 1/q_0 + 1/r_0 < \frac32$.
Also, for any $p\in(0,p_0)$ let $a \in L^{\max(2, p)} [0,T]$,
for any $q\in[1,q_0)$ let $b \in L^q [0,T]$,  and for any $r\in(0,r_0)$ let $c \in L^{\max(1, r)} [0,T]$.

Then the stochastic process
$X$ defined in \eqref{eq:volterra1}
has a modification that satisfies H\"older condition up to order
$\lambda = \frac32 - 1/q_0 - \max(\frac12, \: 1/p_0 + 1/r_0)$.
\end{lemma}

\begin{remark}
    In Remark~\ref{rem1} we applied
    Lemma~\ref{lem:pabc_Holder} and obtained
    that the fBm with Hurst index $H > \frac12$ has a modification
    that satisfies H\"older condition up to order $\frac12$.
    With Lemma~\ref{lem:cor:c2pabc_Holder},
    we can obtain the same result more easily.
    We just apply Lemma~\ref{lem:cor:c2pabc_Holder}
    for $p_0 = \left(H-\frac12\right)^{\!-1}\!$,
    $q_0 = \infty$ and $r_0 = \left(\frac32 - H\right)^{\!-1}$
    and do not bother with $\epsilon$.
\end{remark}
\begin{proof}
Notice that $0 < \lambda \le 1$.
Denote $A = \left\{m \in \mathbb{N} \, : \,
 m > \max\bigl(3, \frac{\lambda q_0}{q_0-1}\bigr)\right\}$
a set of ``large enough'' positive integers.

Let $n \in A$.
Let $p_n$, $q_n$ and $r_n$ be such real numbers that
$1/p_n = \min(\frac12, \: 1/p_0 + \lambda/n)$,
$1/q_n = 1/q_0 + \lambda/n$,
and $1/r_n = \min(1, \: 1/r_0 + \lambda/n)$.
Then $p_n \in \bigl[\frac12, \infty\bigr)$,
$q_n \in (1,\infty)$,
$r_n \in [1,\infty)$,
and $1/p_n + 1/q_n + 1/r_n < \frac32$.
Apply Lemma~\ref{lem:pabc_Holder} for functions $a$, $b$, $c$
and exponents $p_n$, $q_n$ and $r_n$.
By Lemma~\ref{lem:pabc_Holder}, the process $X$ has a modification
$X^{(n)}$ that satisfies Holder condition up to order
$\frac32 - 1/q_n - \max(\frac12,\: 1/p_n + 1/r_n)
\ge (n-3) \lambda / n$.

For different $n\in A$, the processes $X^{(n)}$ coincide almost surely on $[0,T]$.
Let $B$ be a random event which occurs when
all these processes coincide:
\[
B = \{\forall m\mathbin{\in} A \;\,
\forall n\mathbin{\in} A \;\,
\forall t\mathbin{\in}[0,T] \; : \;
X^{(m)}_t = X^{(n)}_t\}.
\]
Then $\pr(B) = 1$, and $\widetilde X = X^{(k)} \indicatorFun_B$
(where $k = \min A$ is the least element
of the set $A$)
is a modification of $X$ that satisfies
H\"older condition
up to order
$\lambda$.
\end{proof}

\begin{lemma}\label{lem:Hoecumno}
Let
$a\in L^p[0,T]$,
$b\in L^q[0,T]$,
$c\in L^r[0,T]$,
where the exponents satisfy
relations
$ p\in[2, \infty] $,
$ q\in[1, \infty) $,
$ r\in[1, \infty] $,
and $1/p + 1/q + 1/r \le \frac32$.
Let there exist $\lambda>0$   and $C\in\mathbb{R}$ such that
\[\forall t\in[0,T]\; : \;
0 \le \|b\indicatorFun_{[0,t]}\!\|_q \le C t^\lambda .
\]
Then the stochastic process
$\{X_t,\allowbreak\; t\in[0,T]\}$ has a modification
which is continuous on $[0,T]$ and satisfies
H\"older condition at point 0
up to order $\lambda$.
\end{lemma}
\begin{remark}
For the fBm with Hurst index $H>\frac12$,
apply Lemma~\ref{lem:Hoecumno}
to the functions $a$, $b$ and $c$
defined in Remark~\ref{rem1},
but for exponents
$1/p = H - \frac12 + \frac{\epsilon}{2}$,
$1/q = \frac12 - \epsilon$,
and  $1/r = \frac32 - H + \frac{\epsilon}{2}$
for some $\epsilon$ such that
$0 < \epsilon < \min\bigl(
2(1-H), \: \allowbreak
\frac12, \: \allowbreak
2\left(H-\frac12\right) \bigr)$.
Verify the conditions of Lemma~\ref{lem:Hoecumno}.
We have $H - \frac12 < 1/p < \frac12$,
$0 < 1/q < 1$,
$\frac32 - H < 1/r < 1$
(whence $a\in L^p[0,T]$ and $c\in L^r[0,T]$;
the relation $b\in L^q[0,T]$ holds true for all $q\ge 1$)
and $1/p + 1/q + 1/r = \frac32$.
Moreover, $\|b \indicatorFun_{[0,t]}\|_q
= C_\epsilon t^{H-1/2+1/q}$,
where $C_\epsilon = ((H-\frac12)q + 1)^{-1/q}$.
According to Lemma~\ref{lem:Hoecumno},
the fBm satisfies H\"older condition
at point $0$ up to order
$H-\frac12+1/q = H - \epsilon$.
As this can be proved for any $\epsilon>0$
small enough,
the fBm satisfies H\"older condition
at point $0$ up to order $H$.
\end{remark}

\begin{proof}
Without loss of generality we can assume that
$1/q + 1/r \ge 1$.
Indeed, under original conditions of the lemma,
let $r' = \min(r,\: q / (q-1))$.
Then $1 \le r' \le r$,
$1/q + 1/r' \ge 1$,
 $1/p + 1/q + 1/r' \le \frac32$,
 and $c\in L^{r'}[0,T]$.
The inequality $1/p + 1/q + 1/r' \le \frac32$ can be proved as follows:
 \begin{gather*}
\frac{1}{p} + \frac{1}{q} + \frac{1}{r'} = \frac{1}{p} + \frac{1}{q} + \frac{1}{r} \le \frac32 \quad \mbox{if} \quad r \le \frac{q}{q-1};\\
\frac{1}{p} + \frac{1}{q} + \frac{1}{r'} = \frac{1}{p} + \frac{1}{q} + \frac{q-1}{q} = \frac{1}{p} + 1 \le \frac12 + 1 = \frac32
\quad \mbox{if} \quad r \ge \frac{q}{q-1}.
\end{gather*}
The other relations can be proved easily.
Thus, after substitution of $r'$ for $r$
all conditions of Lemma~\ref{lem:Hoecumno} still hold true,
as well as $1/q + 1/r \ge 1$.

Denote
\[
F(t) = \int_0^t |b(s)|^q \, dt + t^{\lambda q}.
\]
Then $F:[0,T] \to [0,+\infty)$
is a strictly  increasing function such that
\begin{gather*}
F(0) = 0, \qquad\qquad
F(t) \le C_1 t^{\lambda q}
\quad \mbox{if\quad $0\le t \le T$},\\
\|b \indicatorFun_{(t_1,t_2]}\!\|_q
< (F(t_2) - F(t_1))^{1/q}
\quad \mbox{if\quad $0\le t_1 < t_2 \le T$}.
\end{gather*}

Let $0 \le t_1 < t_2 \le T$.
Again, denote $\tilde c(v) = c(-v)$.
Let us construct an upper bound for
$\|K_1(t_2,\,\cdot\,) - K_1(t_1,\,\cdot\,)\|_2 =
\| a \, (b \indicatorFun_{(t_1,t_2]} * \tilde c) \|_2
$,
see \eqref{eq:difKt2sKt1s}.
By Young's convolution inequality \eqref{neq:Young},
\[
\|b \indicatorFun_{(t_1,t_2]} * \tilde c\|
_{ (1/q+1/r-1)^{-1}} \le
\|b \indicatorFun_{(t_1,t_2]}\!\|_q \, \|\tilde c\|_r
\le (F(t_2) - F(t_1))^{1/q} \, \|c\|_r .
\]
Here we used that $q\ge 1$, $r \ge 1$ and $1/r+1/q \ge 1$.

The function $a \, (b \indicatorFun_{(t_1,t_2]} * \tilde c)$
is equal to $0$ outside the interval $[0, t_2]$.
Noticing that $2 \le (1/p + 1/q + 1/r - 1)^{-1}$,
using the inequality \eqref{neq:LL} for norms in $L^2[0, t_2]$
and $L^{(1/p + 1/q + 1/r - 1)^{-1}}[0, t_2]$
and H\"older inequality for non-conjugate exponents \eqref{neq:Holder},
we get
\begin{align*}
\| a \, (b \indicatorFun_{(t_1,t_2]} * \tilde c) \|_2
&\le
\| a \, (b \indicatorFun_{(t_1,t_2]} * \tilde c) \|_{ (1/p+1/q+1/r-1)^{-1}}\,
t_2^{\frac32 - 1/p - 1/q - 1/r}
\\ &
\le \|a\|_p \,
\|b \indicatorFun_{(t_1,t_2]} * \tilde c\|_{ (1/q+1/r-1)^{-1}}\,
t_2^{\frac32 - 1/p - 1/q - 1/r}
\\ &\le
\|a\|_p \,
(F(t_2) - F(t_1))^{1/q} \, \|c\|_r \,
t_2^{\frac32 - 1/p - 1/q - 1/r}.
\end{align*}
Hence
\begin{align*}
\ex{\bprod{X_{t_2}-X_{t_1}}^2} &= \|K(t_2,\,\cdot\,)-K(t_1,\,\cdot\,)\|_2^2
= \| a \, (b \indicatorFun_{(t_1,t_2]} * \tilde c) \|_2^2 \\
&\le
\|a\|_p^2 \,
(F(t_2) - F(t_1))^{2/q} \, \|c\|_r^2 \,
t_2^{3 - 2(1/p + 1/q + 1/r)} .
\end{align*}

Consider stochastic process $Y = \{Y_s : s\in[0, F(T)]\}$,
with $Y_{F(t)} = X_t$ for all $t\in[0,T]$.
This process $Y$ satisfies inequality
\[
\ex{\bprod{Y_{s_2}-Y_{s_1}}^2}
\le
\|a\|_p^2 \,
(s_2 - s_1)^{2/q} \, \|c\|_r^2
T^{3 - 2(1/p + 1/q + 1/r)}
\quad \mbox{if $0\le s_1 < s_2 \le F(T)$}.
\]
By Corollary~\ref{cor:Gausscontp},
the process $Y$
has a modification $\widetilde Y$ that satisfies H\"older condition
up to order $1/q$.
Therefore, for any $\lambda_1 \in (0, \lambda)$
\[
\exists C_2 \; \forall s_1 \! \in [0,F(T)] \;\,
\forall s_2 \mathrel{\in} [0,F(T)] \; : \;
|\widetilde Y_{s_2} - \widetilde Y_{s_1}| \le C_2 \, |s_2 - s_1|^{\lambda_1  / (\lambda q)} ,
\]
where $C_2$ is a random variable; $C_2 < \infty$ surely.
In particular,
\[
\exists C_2 \;
\forall s \mathrel{\in} [0,F(T)] \; : \;
|\widetilde Y_{s} - \widetilde Y_0| \le C_2 \, s^{\lambda_1   / (\lambda q)} .
\]
The stochastic process
$\widetilde X = \{\widetilde X_t, \allowbreak\; t \in [0,T]\}
= \{\widetilde Y_{F(t)}, \allowbreak\; t\in[0,T]\}$
is a modification of the stochastic process $X$.
It satisfies inequalities
\begin{gather*}
\exists C_2 \; \forall t \mathrel{\in} [0,T] \; : \;
|\widetilde X_t - \widetilde X_0| \le C_2 \, F(t)^{\lambda_1  / (\lambda q)}  ; \\
\exists C_3 \; \forall t \mathrel{\in} [0,T] \; : \;
|\widetilde X_t - \widetilde X_0| \le C_3 \, t^{\lambda_1} .
\end{gather*}
Thus, all the paths of the stochastic process $\widetilde X$
satisfy H\"older condition at point 0
with exponent $\lambda_1$.
\end{proof}

\section{Gaussian Volterra processes with Sonine kernels}\label{sec-3}
\subsection{Fractional Brownian motion and  Sonine kernels}\label{sec-fBm-Sonine}
Consider now a natural question: for which kernels $K$ of the form \eqref{eq:abckernel} Gaussian process of the form \eqref{eq:volterra1} with Volterra kernel $K$ generates the same filtration as the Wiener process $W$. Sufficient condition for this is the representation of the Wiener process $W$ as  \begin{equation}\label{eq:volterra.inverse}
W_t = \int_0^t L(t,s) \, dX_s
\end{equation}
where $L\in L^2([0,T]^2)$ is a Volterra kernel, and the integral is well defined, in some sense.
 As an example, let us consider fractional Brownian motion $B^H, H>1/2 $ admitting a  representation \eqref{eq:volterra} with Volterra kernel \eqref{equ-bigH}. For any $0<\varepsilon<1$  consider the approximation
$$B^{H,\varepsilon}_t=d_H\int_0^t
\left(s^{1/2 - H}
\int_s^t u^{H-1/2} (u-\varepsilon s)^{H-3/2}\, du\right) dW_s, t\ge 0.$$ Unlike the original process, in such approximation  we can change the limits of integration and get that

$$B^{H,\varepsilon}_t=d_H\int_0^t \left(u^{H-1/2}\int_0^u
s^{1/2 - H}
  (u-\varepsilon s)^{H-3/2}dW_s\, \right)du.$$ This representation allows to    write the equality
\begin{equation}\label{equ-int}\int_0^tu^{1/2-H}dB^{H}_u=d_H\int_0^t \left( \int_0^u
s^{1/2 - H}
  (u-\varepsilon s)^{H-3/2}dW_s\, \right)du,\end{equation}
  and it follows immediately from \eqref{equ-int} that
  \begin{equation}\begin{gathered}\label{equ-int1}\int_0^t(t-u)^{1/2-H}u^{1/2-H}dB^{H,\varepsilon}_u\hspace{4.5cm}\\=d_H\int_0^t (t-u)^{1/2-H}\left( \int_0^u
s^{1/2 - H}
  (u-\varepsilon s)^{H-3/2}dW_s\, \right)du\\= d_H\int_0^t s^{1/2 - H}\left( \int_s^t(t-u)^{1/2-H}
(u-\varepsilon s)^{H-3/2}du\,\, \right)dW_s.\end{gathered}\end{equation}
Applying  Theorem 3.3  from \cite{BannaMono2019}, p.~160,  we can go to the limit in \eqref{equ-int1} and get that
$$\int_0^t(t-u)^{1/2-H}u^{1/2-H}dB^{H}_u=d_H\int_0^t s^{1/2 - H}\left( \int_s^t(t-u)^{1/2-H}
(u-  s)^{H-3/2}du\,\, \right)dW_s. $$
Now the highlight is that the integral $\int_s^t(t-u)^{1/2-H}
(u-  s)^{H-3/2}\,du$ is a constant, namely, $\int_s^t(t{-}u)^{1/2-H}
(u{-} s)^{H-3/2}du=\int_0^t(t{-}u)^{1/2-H}
 u^{H-3/2}du=\mathrm{B}(3/2{-}H, H{-}1/2)$, where $\mathrm{B}$ is a beta-function. After we noticed this, then everything is simple:
$$Y_t:=\int_0^t(t-u)^{1/2-H}u^{1/2-H}dB^{H}_u=d_H \mathrm{B}(3/2-H,\: H-1/2)\int_0^t s^{1/2 - H} dW_s, $$ and finally  we get that
$$W_t=e_H\int_0^t s^{  H-1/2} dY_s $$ with some constant $e_H$. It means that we have representation \eqref{eq:volterra.inverse}  and, in particular, $W$ and $B^H$ generate the same filtration.
Of course, these transformations can be performed much faster, but our goal here was to pay attention on the role of the property of the convolution of two functions to be a constant. This property is a characterization of Sonine kernels.
  \subsection {General approach to    Volterra processes with Sonine kernels}

  First we give basic information about Sonine kernels, more details can be found in \cite{SC03}. We also consider, in a simplified form, the related  generalized fractional calculus introduced in \cite{kochubei}.

\begin{definition} A function $c\in L^1[0,T]$ is called a Sonine kernel  if there exists a function $h\in L^1[0,T]$ such that
\begin{equation}\label{eq:sonine}
\int_0^t c(s) h(t-s) \, ds= 1,\quad t\in(0,T].
\end{equation}
Functions $c,h$ are called  Sonine pair, or, equivalently, we say that $c$ and $h$ form (or create) a Sonine pair.
\end{definition}
If $\hat c$ and $\hat h$ denote the Laplace transforms of $c$ and $h$ respectively, then \eqref{eq:sonine} is equivalent to $\hat c(\lambda)\hat h(\lambda) = \lambda^{-1}$, $\lambda>0$.
Since the Laplace transform characterizes a function uniquely, for any $c$ there can be not more than one function $h$ satisfying \eqref{eq:sonine}. Examples of Sonine pairs are given   in Section \ref{sec:examples}.

Let functions $c$ and $h$ form a Sonine pair. For a function $f\in L^1[0,T]$ consider the operator
$$
\mathrm I^c_{0+} f (t) =\int_0^t c(t-s) f(s) ds.
$$
It is an analogue of forward fractional integration operator. Let us identify an inverse operator.
In order to do this, for $g \in AC[0,T]$ define
$$
\mathrm D^h_{0+} g(t) = \int_0^t h(t-s)g'(s)ds + h(t)g(0).
$$

Note that
\begin{gather*}
\int_0^t \mathrm D^h_{0+} g(u) du = \int_0^t \left(\int_0^u h(u-s)g'(s)ds + h(u)g(0)\right)\, du \\ = \int_0^t \int_0^u h(s)g'(u-s)ds\, du + g(0)\int_0^t h(u)du
\\=\int_0^t h(s)\int_s^t g'(u-s)du\, ds + g(0)\int_0^t h(u)du \\
 =\int_0^t h(s)\big(g(t-s)-g(0)\big)ds\, ds + g(0)\int_0^t h(u)du = \int_0^t h(s) g(t-s)ds,
\end{gather*}
so we can also write
\begin{equation}\label{eq:derivative}
\mathrm D^h_{0+} g(t) = \frac{d}{dt}\int_0^t h(s) g(t-s)ds = \frac{d}{dt}\int_0^t h(t-s) g(s)ds,
\end{equation}
where the derivative is understood in the weak sense.
Similarly, we can define an analogue of backward fractional integral:
$$
\mathrm I^c_{T-} f(s) = \int_s^T c(t-s) f(t)dt, \qquad
f \in L^1[0,T]
$$
and the corresponding differentiation operator
$$
\mathrm D^h_{T-} g(s) = g(T)h(T-s) - \int_s^T h(t-s)g'(t)dt.
$$
\begin{lemma}\label{lem:forwarddiff}
Let $g\in AC[0,T]$. Then $\displaystyle\mathrm{I}_{0+}^c\big( \mathrm D^h_{0+} g\big)(t) = g(t)$  and $\displaystyle\mathrm{I}_{T-}^c\big( \mathrm D^h_{T-} g\big)(s) = g(s)$.
\end{lemma}
\begin{proof}
We have
\begin{align*}
\mathrm I^c_{0+}\big( \mathrm D^h_{0+} g\big)(t)
&=
\int_0^t c(t-s)\left(\int_0^s h(s-u) g'(u)du + h(s)g(0)\right)\, ds
\\ &=
\int_0^t \int_u^t c(t-s)h(s-u)ds\, g'(u)du + g(0)\int_{0}^t c(t-s)h(s)ds
\\ & =
\int_0^t  g'(u)du + g(0) = g(t),
\end{align*}
as required.
Similarly,
\begin{align*}
\mathrm I^c_{T-}\big( \mathrm D^h_{T-} g\big)(s) &= \int_s^T c(t-s)\left(h(T-t)g(T) - \int_t^T h(u-t) g'(u)du \right) ds\\
 &= g(T)\int_{s}^T c(t-s)h(T-t)dt -  \int_s^T \int_s^u c(t-s)h(u-t)\,dt\, g'(u)\, du \\
  &= g(T) - \int_s^T  g'(u)du + g(s) = g(s)
\end{align*}
as required.
\end{proof}
Now consider a Gaussian process  $X$  given by the integral transformation of type \eqref{eq:volterra} of the form \eqref{eq:abckernel} satisfying condition \ref{cond:K1} of Theorem~\ref{thm:lem_XEcontin}.
Define the integral operator
$$
\mathcal K f(t) = \int_0^t a(s)\int_s^t b(u)c(u-s) du\, f(s) ds.
$$
Note that for $f\in L^2[0,T]$, $\mathcal K f(t)\in AC[0,T]$.
Indeed, by definition,
$$
\mathcal K f (t) = \int_0^t K(t,s) f(s)ds =  \int_0^t \int_s^t \frac{\partial}{\partial u}K(u,s) du\, f(s)ds.
$$
Since $f$ and $\frac{\partial}{\partial t}K(t,s)$ are square integrable, the product $f \frac{\partial}{\partial u}K$ is integrable on $\set{(s,u): 0\le s\le u\le t }$. Therefore, we can apply Fubini theorem to get
$$
\mathcal K f (t) = \int_0^t \int_0^u \frac{\partial}{\partial u}K(u,s) f(s) ds\, du = \int_0^t \alpha(u)du,
$$
where $\alpha \in L^1[0,t]$ for all $t\in[0,T]$, so $\alpha \in L^1[0,T]$.
Consequently, for $f\in L^2[0,T]$ we can denote by $$\mathcal J f(t) = \int_0^t \frac{\partial}{\partial t}K(t,s) f(s) ds$$ the weak derivative of $\mathcal K f$.

Further, define for a measurable $g\colon [0,T] \to \R$ such that
$$
\norm{g}_{\mathcal H_X}^2 := \int_0^T \left(\int_s^T \frac{\partial}{\partial u} K(u,s) g(u) du\right)^2 ds<\infty
$$
the integral operator
$$
\mathcal J^* g(s) = \int_s^T \frac{\partial}{\partial u} K(u,s) g(u)du.
$$
It can be extended to the completion $\mathcal H_X$ of the set  of measurable functions with finite norm $\norm{\cdot}_{\mathcal H_X}^2$ so that
$$
\norm{g}_{\mathcal H_X}^2  = \int_0^T \bigl(\mathcal J^* g (t)\bigr)^2 dt,\ g\in \mathcal H_X.
$$
The operator $\mathcal J^*$ is related to the adjoint $\mathcal K^*$ of $\mathcal K$ in the following way: for a finite signed measure $\mu$ on $[0,T]$,
$$
\mathcal K^* \mu  = \mathcal J^* h
\quad \mbox{with}
\quad h(t) = \mu([t,T]).
$$

We are going to identify inverse to the operators $\mathcal J$ and $\mathcal J^*$.
Clearly, it is not possible in general, so we will assume that
{\settowidth{\leftmargini}{(K2)~\,}
\begin{enumerate}
\renewcommand{\theenumi}{{\upshape(S)}}
\renewcommand{\labelenumi}{\upshape(S) \,\,}
	\item\label{cond:S}
the function $c$ forms a Sonine pair with some $h\in L^1[0,T]$.
\end{enumerate}}
In this case the operators $\mathcal J$ and $\mathcal J^*$ can be written in terms of ``fractional'' operators defined above:
$$
\mathcal J f(t) = \int_0^t \frac{\partial}{\partial t}K(t,s) f(s)ds =
\int_0^t a(s) \, b(t) \, c(t-s) f(s) \, ds
= b(t)\, \mathrm I_{0+}^c (af)(t),
$$
and
$$
\mathcal J^* g(s) =
\int_s^T \! a(s) \, b(t) \, c(t-s) \, g(t) \, dt =
a(s)\,\mathrm I_{T-}^c (bg)(s).
$$
In order for this operators to be injective, we assume
{\settowidth{\leftmargini}{(K2)~\, }
\begin{enumerate}
\renewcommand{\theenumi}{{\upshape(K\arabic{enumi})}}
\renewcommand{\labelenumi}{\upshape(K\arabic{enumi})}
\setcounter{enumi}{1}
	\item\label{cond:K2}
the functions $a,b$ are positive a.e.\ on $[0,T]$.
\end{enumerate}}
For $f$ such that $fb^{-1}\in AC[0,T]$, define
$$
\mathcal L f(t) = a(t)^{-1}\mathrm D_{0+}^h \big(f b^{-1}\big)(t) = a(t)^{-1}\left(\int_0^t h(t-s) \big(fb^{-1}\big)'(s)ds + h(t) \big(fb^{-1}\big)(0)\right),
$$
and for $g$ such that $ga^{-1}\in AC[0,T]$, define
\begin{align*}
\mathcal L^* g(s)
&=
b(s)^{-1}\mathrm D_{T-}^h \big(g a^{-1}\big)(s)
\\ &=
b(s)^{-1}\left(h(T-s) \big(ga^{-1}\big)(T) - \int_s^T h(t-s)\big(ga^{-1}\big)'(t)dt\right).
\end{align*}
\begin{proposition}\label{prop:Kinv}
Let the assumptions \ref{cond:S}, \ref{cond:K1} and \ref{cond:K2} hold.
Then the operators $\mathcal J$ and $\mathcal J^*$ are injective,
and for functions $f,g$ such that $fb^{-1}\in AC[0,T]$, $ga^{-1}\in AC[0,T]$,
$$
\mathcal J \mathcal L f(t) = f(t), \qquad \mathcal J^* \mathcal L^* g(s) = g(s).
$$
\end{proposition}
\begin{proof}
Assume that $\mathcal J f = 0$ for some $f\in L^2[0,T]$. Then, by \ref{cond:K2}, $\mathrm I_{0+}^c (bf) = 0$ a.e.\ on $[0,T]$. Therefore, for any $t\in[0,T]$
\begin{align*}
0 &= \int_0^t h(t-s)\mathrm I_{0+}^c (bf)(s)ds = \int_0^t h(t-s)\int_0^s c(s-u) b(u)f(u)du \,ds \\
&= \int_0^t \int_u^t h(t-s)c(s-u)ds\,  b(u)f(u) du = \int_0^t  b(u)f(u) du,
\end{align*}
whence $b f = 0$ a.e.\ on $[0,T]$, so, applying to \ref{cond:K2} once more, $f = 0$ a.e.{\spacefactor=3000}
The injectivity of $\mathcal J^*$ is shown similarly, and the second statement follows from Lemma \ref{lem:forwarddiff}.
\end{proof}
Now we are in a position to invert the covariance operator $\mathcal R = \mathcal K\mathcal K^*$ of $X$. We need a further assumption.
{\settowidth{\leftmargini}{(K2)~\,}
\begin{enumerate}
\renewcommand{\theenumi}{{\upshape(K\arabic{enumi})}}
\renewcommand{\labelenumi}{\upshape(K\arabic{enumi})}
\setcounter{enumi}{2}
	\item\label{cond:K3}
$a^{-1}\in C^1[0,T]$,\; $d:= b^{-1}\in C^2[0,T]$ and either  $d(0) = d'(0) = 0$ or\\ $a^{-2}h\in C^1[0,T]$.
\end{enumerate}}
\begin{proposition}\label{prop:Rinv}
Let the assumptions {\ref{cond:S}, \ref{cond:K1}\,--\,\ref{cond:K3}} hold, and $f\in C^3[0,T]$ with $f(0)= 0$. Then for $h = \mathcal L^*\mathcal L f'$, the measure $\mu([t,T]) = h(t)$
 is such that $\mathcal R \mu = f$.
\end{proposition}
\begin{proof}
Thanks to \ref{cond:K3}, $f'b^{-1}\in AC[0,T]$ and
\begin{equation}\label{eq:aLf}
a(t)^{-1}\mathcal L f'(t) = a(t)^{-2}\left(\int_0^t h(t-s) \big(fb^{-1}\big)'(s)ds + h(t) \big(f'b^{-1}\big)(0)\right).
\end{equation}
Similarly to \eqref{eq:derivative}, $\int_0^t h(t-s) \big(fb^{-1}\big)'(s)ds$ is absolutely continuous with
$$
\frac{d}{dt}\int_0^t h(t-s) \big(f'b^{-1}\big)'(s)ds = \int_0^t h(s) \big(f'b^{-1}\big)''(t-s)ds + h(t) \big(f'b^{-1}\big)'(0).
$$
Then, thanks to \ref{cond:K3}, both summands in the right-hand side of \eqref{eq:aLf} are absolutely continuous with bounded derivatives. So by Proposition~\ref{prop:Kinv},
\[
\mathcal K^* \mu = \mathcal J^* h = \mathcal J^* \mathcal L^* \mathcal L f'
= \mathcal L f' .
\]
%$$
%\mathcal K^*h = \mathcal K^* \mathcal L^* (\mathcal L f') = \mathcal Lf'
%$$
and
\[
\mathcal J \mathcal K^* \mu = \mathcal J \mathcal L f' = f' .
\]
%$$
%\mathcal K' \mathcal K^* h = \mathcal K'\mathcal Lf' = f'.
%$$
Therefore,
\[
\mathcal R \mu \,(t) =
\mathcal K \mathcal K^* \mu \, (t) =
\int_0^t \mathcal J \mathcal K^* \mu \, (s) \, ds =
\int_0^t f'(s) \, ds = f(t)
\]
%$$
%\mathcal R h(t) = \int_0^t\mathcal K' \mathcal K^* h(s)ds = \int_0^t f'(s)ds= f(t),
%$$
as required.
\end{proof}
Now we recall the definition of integral with respect to the $X$ given by \eqref{eq:volterra1}; for more details see \cite{AMN01}. Define
$$
I_X(\ind{[0,t]}) = \int_0^T \ind{[0,t]}(s) \, dX_s  = X_t
$$
and extend this by linearity to the set $\mathcal S$ of piecewise constant function. Then, for any $g\in \mathcal S$,
$$
\ex{I_X(g)^2} = \norm{g}^2_{\mathcal H_X}.
$$
Therefore, $I_X$ can be extended to isometry between $\mathcal H_X$ and a subspace of $L^2(\Omega)$. Moreover, for any  $g\in \mathcal H_X$,
\begin{equation}\label{intdX}
	\int_0^T \! g(t) \, dX_t = \int_0^T \! \mathcal J^* g(t) \, dW_t.
\end{equation}
\begin{proposition}
	Let the assumptions $(\mathrm{S})$, $(\mathrm{K}1)-(\mathrm{K}3)$ be satisfied, and $X$ be given by \eqref{eq:volterra1}. Then
	$$W_t = \int_0^t k(t,s) \, dX_s,
	$$
	where
	$$
	k(t,s) = p(t)b(s)^{-1} h(t-s) - b(s)^{-1}\int_s^t p'(v) h(v-s) \, dv ,
	$$
	and $p = a^{-1}$.
\end{proposition}
\begin{proof}
	Write $k(t,s) = k_1(t,s) - k_2(t,s)$, where $ k_1(t,s) = p(t)b(s)^{-1} h(t-s),  k_2(t,s) =b(s)^{-1}\int_s^t p'(v) h(v-s) dv$, and transform
	\begin{align*}
			\bigl(\mathcal J^* k_1(t,\cdot)\ind{[0,t]}\bigr)(s)  & =
	\int_s^T \frac{\partial}{\partial u}K(u,s) p(t)b(u)^{-1} h(t-u)\ind{[0,t]}(u)du \\
		& = p(t)\int_s^t a(s)b(u)c(u-s) b(u)^{-1}h(t-u)du\, \ind{[0,t]}(s) \\
		& = p(t)\,a(s)\int_s^t c(u-s) h(t-u)du\, \ind{[0,t]}(s)\\
		& = p(t)\,a(s)\ind{[0,t]}(s).
	\end{align*}
	Similarly,
	\begin{align*}
		\bigl(\mathcal J^* k_2(t,\cdot)\ind{[0,t]}\bigr)(s)  & =
		\int_s^t a(s)\,c(u-s)\int_u^t p'(v) h(v-u)dv\, du\,\ind{[0,t]}(s)\\
		& = a(s)\int_s^t p'(v) \int_s^v c(u-s)h(v-u)du\, dv \,\ind{[0,t]}(s)\\
		& = a(s)\int_s^t p'(v) \, dv \,\ind{[0,t]}(s) = a(s)\bigl(p(t) - p(s)\bigr)\ind{[0,t]}(s).
	\end{align*}
	Consequently,
	\begin{align*}
	\bigl(\mathcal J^* k(t,\cdot)\ind{[0,t]}\bigr)(s)  & = p(t)\,a(s)\ind{[0,t]}(s) - a(s)\bigl(p(t) - p(s)\bigr)\ind{[0,t]}(s)\\
	& = a(s) \, p(s)\ind{[0,t]}(s) = \ind{[0,t]}(s).
	\end{align*}
	Therefore, thanks to \eqref{intdX},
	\begin{align*}
	\int_0^T k(t,s) \, dX_s = \int_0^T \bigl(\mathcal J^* k(t,\cdot)\ind{[0,t]}\bigr)(s) \, dW_s = \int_0^T \ind{[0,t]}(s) \, dW_s = W_t,
	\end{align*}
	as required.
\end{proof}

\section{Examples of Sonine kernels}\label{sec:examples}

\begin{example} Functions $c(s)=s^{-\alpha}$ and  $h(s)=s^{ \alpha-1}$ with some $\alpha\in (0,1/2)$  were considered above in connection with fractional Brownian motion, see subsection \ref{sec-fBm-Sonine}.
\end{example}
\begin{example}
For $\alpha\in(0,1)$ and $A\in \R$, let $\gamma=\Gamma'(1)$ be Euler-Mascheroni constant, $l = \gamma -A$. Then
$$c(x) = \frac{1}{\Gamma(\alpha)} x^{\alpha-1}\left(\ln \textstyle{\frac1x} + A\right)$$
and
$$
h(x) = \int_0^\infty \frac{x^{t-\alpha}e^{lt}}{\Gamma(1-\alpha + t)}dt
$$
create a  Sonine pair, see \cite{SC03}.
\end{example}

\begin{example}\label{2.10}
This example was proposed by Sonine himself \cite{sonin}: for $\nu\in(0,1)$,
$$
h(x) = x^{-\nu/2}J_{-\nu}(2\sqrt{x}), \quad c(x) = x^{(\nu-1)/2} I_{\nu-1}(2\sqrt{x}),
$$
where $J$ and $I$ are, respectively, Bessel and modified Bessel functions of the first kind,
\begin{equation*}
  J_\nu(y)=\frac{y^\nu}{2^\nu}\sum_{k=0}^\infty
  \frac{(-1)^ky^{2k}2^{-2k}}{k!\Gamma(\nu+k+1)},
\end{equation*} and
\begin{equation*}
  I_\nu(y)=\frac{y^\nu}{2^\nu}\sum_{k=0}^\infty
  \frac{ y^{2k}2^{-2k}}{k!\Gamma(\nu+k+1)}.
\end{equation*} In particular, setting $\nu = 1/2$, we get the following Sonine pair:
\begin{gather}\label{Sonine}
h(x) = \frac{\cos 2\sqrt{x}}{2\sqrt{\pi x}},\quad c(x) = \frac{\cosh 2\sqrt{x}}{2\sqrt{\pi x}}.
\end{gather}
\end{example}
\begin{remark} It is interesting that the creation of Sonine pairs allows to get the relations between the special functions (see \cite[Section 1.14]{bookmyus}).
Let
\begin{gather*}
c(x)=x^{-1/2}\cosh(ax^{1/2}),\end{gather*}
and let \begin{gather*}h(x)=\int_0^xs^{\nu/2}
J_\nu(as^{1/2})\,(x-s)^\gamma ds
\end{gather*} be a fractional integral of $s^{\nu/2}
J_\nu(as^{1/2})$,
where $-1<\nu<-\frac{1}{2}$,
$\gamma+\nu=-\frac{3}{2}$.
If we denote $F_y(\lambda)$  Laplace transform of function $y$ at point $\lambda$, then
the Laplace transforms of these functions equal
\begin{gather*}
F_c(\lambda)=(\pi/\lambda)^{1/2}\exp (a^2/4\lambda),
\\
\begin{aligned}
F_{h}(\lambda)&=\Gamma(\gamma+1)2^{-\nu}a^\nu\lambda^{-\nu-1} \exp(-a^2/4\lambda)\lambda^{-\gamma-1} \\
&=\Gamma(\gamma+1)2^{-\nu}a^\nu\lambda^{-1/2}
\exp(-a^2/4\lambda),
\end{aligned}\\
F_c(\lambda)
F_{h}(\lambda)=\Gamma(\gamma+1)2^{-\nu }\sqrt{\pi} a^\nu\lambda^{-1},\qquad
\lambda>0,
\end{gather*}
whence their convolution equals
$$(c\ast h)_t=\Gamma(\gamma+1)2^{-\nu }\sqrt{\pi}
a^\nu,  \qquad t>0.$$
Therefore $c(x)$ and $(\Gamma(\gamma+1)2^{-\nu }\sqrt{\pi}
a^\nu)^{-1}h(x)$ create a Sonine pair. However, comparing with  Example \ref{2.10} with $a=2$, and taking into account that the pair in Sonine pair is unique,
we get that $$ 4\sqrt{\pi}(\Gamma(\gamma+1)
 )^{-1}\int_0^xs^{\nu/2}
J_\nu(2s^{1/2})\,(x-s)^\gamma ds=\frac{\cos 2\sqrt{x}}{  \sqrt{  x}}.$$
Similarly, let    $c(x)=\int_0^xt^{-1/2}\cosh(at^{1/2})\,(x-t)^\gamma dt,\
h(x)=x^{\nu/2}J_\nu(ax^{1/2})$ with $\gamma\in(-1,-\frac{1}{2})$,\spacefactor=3000{} 
$\nu\in(-1,0)$,,\spacefactor=3000{} $\gamma+\nu=-\frac{3}{2}$. Then $$F_{c}(\lambda)
=\pi^{1/2}\Gamma(\gamma+1)\lambda^{-\gamma-3/2}\exp(a^2/4\lambda),$$ and
$$F_h(\lambda)=\frac{a^\nu}{2^\nu}\lambda^{-\nu-1}\exp(-a^2/4\lambda),\;\text{whence}\;
F_{c}(\lambda)F_h(\lambda)=\pi^{1/2}\Gamma(\gamma+1)\frac{a^\nu}{2^\nu}\lambda^{-1}.$$
If we put $a=2$ and compare with \eqref{Sonine}, we get the following representation
$$\pi^{-1/2}(\Gamma(\gamma+1))^{-1}\int_0^xt^{-1/2}\cosh(2t^{1/2})\,(x-t)^\gamma dt=x^{(-\nu-1)/2} I_{-\nu-1}(2\sqrt{x}).$$
\end{remark}

\begin{example}\label{example:positive}
On the way of creation of the new Sonine pairs, a natural idea is to consider $g(s)=e^{\beta s}s^{ \alpha-1}$ with $\beta\in\R$ and examine if this function admits a Sonine pair.  It happens so  that the answer to this question is positive, but far from obvious and not simple. All preliminary results are contained in subsection \ref{subsect-4}.
Let
\[
g(x) = \frac{\exp(\beta x)}{\Gamma(\alpha) x^{1-\alpha}}
, \quad 0<\alpha<1, \quad \beta<0;
\qquad
y(x) = 1.
\]
Then
\[
h(x) = \alpha \beta \hyperIFI(\alpha+1; \: 2; \: \beta x)
< 0, \qquad x\in[0,T],
\]
where $\hyperIFI$ is Kummer hypergeometric
function; see Section~\ref{sect:Example} in the Appendix.  
 The conditions of Theorem \ref{thm:positive} hold  true.
The equation \eqref{eq:Volterra1y1} has
a unique solution in $L^{1}[0,T]$
(Actually, it has many solutions,
but each two solutions are equal almost everywhere.)
The solution has a representative that
is continuous and attains only positive values
on the left-open interval $(0,T]$, and it is a Sonine pair to $g(s)=e^{\beta s}s^{ \alpha-1}$.
\end{example}

\section{Appendix}\label{app}

\subsection{Inequalities for norms of convolutions and products}
Recall notation $\|f\|_p$ for the norm
of function $f\in L^p(\mathbb{R})$,
$p \in [1, \infty]$.
The convolution of two measurable functions
$f$ and $g$ is defined by integration
\begin{equation}\label{eq:defconv}
(f*g)(t) = \int_{\mathbb{R}} f(s) g(t-s) \, ds .
\end{equation}

Now we state an inequality for the norm of convolution
of two functions.
If $p\in[1,\infty]$, $q\in[1,\infty]$
but $1/p+1/q\ge 1$,
$f\in L^p(\mathbb{R})$,
$g\in L^q(\mathbb{R})$,
then the  convolution $f*g$
is well-defined almost everywhere
(that is the integral in
\eqref{eq:defconv} converges absolutely
for almost all $t\in\mathbb{R}$),
$f * g \in L^{(1/p+1/q-1)^{-1}}(\mathbb{R})$,
and
\begin{equation}\label{neq:Young}
\|f*g\|_{(1/p+1/q-1)^{-1}}
\le \|f\|_p \, \|g\|_q.
\end{equation}

Now we state an inequality for the norm of the
product of two functions $(fg)(t) = f(t) g(t)$. We call it H\"older inequality
for non-conjugate exponents.
If $p\in[1,\infty]$, $q\in[1,\infty]$,
$1/p+1/q\le 1$,
$f\in L^p(\mathbb{R})$,
$g\in L^q(\mathbb{R})$,
then $fg \in L^{(1/p+1/q)^{-1}}(\mathbb{R})$
and
\begin{equation}\label{neq:Holder}
\|fg\|_{(1/p+1/q)^{-1}} \le \|f\|_p \, \|g\|_q.
\end{equation}

Now we state an inequality for the norms
in $L_p[a,b]$ and $L_q[a,b]$.
If $-\infty < a < b < \infty$,
$1 \le p \le q \le \infty$,
$f \in L^q(\mathbb{R})$ and
the $f(t)=0$ for all $t\not\in[a,b]$,
then $f\in L^p(\mathbb{R})$ and
\begin{equation}\label{neq:LL}
\|f\|_p \le (b-a)^{1/p-1/q} \, \|f\|_q .
\end{equation}

\begin{remark}
Conditions for inequalities \eqref{neq:Holder}
and \eqref{neq:LL} are over-restrictive
because of restrictive notation $\|f\|_p$.
This notation can be extended to all $p\in(0,\infty]$
and all measurable functions $f$.
Then the conditions for inequalities \eqref{neq:Holder} and \eqref{neq:LL}
may be relaxed.
 \end{remark}

Inequality \eqref{neq:Young} is proved
in \cite[Theorem 4.2]{LiebLoss};
see item (2) in the remarks after this theorem and part  (A) of its proof.
If $p<\infty$ and $q<\infty$, then
inequality \eqref{neq:Holder} follows
from the conventional H\"older inequality.
Otherwise, if $p=\infty$ or $q=\infty$,
then inequality \eqref{neq:Holder} is trivial.
Inequality \eqref{neq:LL} can be rewritten as
$\|f \indicatorFun_{[a,b]}\!\|_p \le
\|\!\indicatorFun_{[a,b]}\!\|_{(1/p-1/q)^{-1}} \, \|f\|_q $,
and so follows from \eqref{neq:Holder}.

\subsection{Continuity of trajectories and H\"{o}lder condition}
Kolmogorov continuity theorem provides sufficiency conditions
for a stochastic process to have a continuous modification.
The following theorem aggregates Theorems 2, 4 and 5 in
\cite{Bell2015}.
\begin{theorem}[Kolmogorov continuity theorem]
Let $\{X_t,\; t\in[0,T]\}$ be a stochastic process.
If there exist $K \ge 0$,
$\alpha>0$ and $\beta>0$ such that
\[
\ex{|X_t - X_s|^\alpha} \le K\, |t-s|^{1+\beta}
\quad \mbox{for all} \quad
0 \le s \le t \le T,
\]
then
\begin{enumerate}
\item
The process $X$ has a continuous modification;
 \item
Every continuous modification of the process $X$
 whose trajectories
almost surely satisfies H\"older condition
for all exponents  $\gamma \in (0,\: \beta/\alpha)$;
\item
There exists a modification of the process $X$
that satisfies H\"{o}lder condition for exponent
$\gamma \in (0, \: \beta/\alpha)$.
 \end{enumerate}
\end{theorem}

This theorem can be applied for Gaussian processes.
\begin{corollary}\label{cor:Gausscontp}
Let $\{X_t,\; t\in[0,T]\}$ be a centered Gaussian process.
If there exist $K\ge 0$ and  $\delta > 0$ such that
\[
\ex{(X_t - X_s)^2} \le K \, |t-s|^\delta
\quad \mbox{for all} \quad
0 \le s \le t \le T,
\]
then the following holds true:
\begin{enumerate}
\item The process $X$ has a modification
$\widetilde X$ that has continuous trajectories.
\item For every $\gamma$, $0< \gamma < \frac{1}{2}\delta$,
the trajectories of the process $\widetilde X$
satisfy $\gamma$-H\"{o}lder condition almost surely.
\item The process $X$ has a modification
that satisfies H\"{o}lder condition for all
exponents $\gamma\in(0, \frac{1}{2}\delta)$.
\end{enumerate}
\end{corollary}

Since $X_s - X_t$ is a centered Gaussian variable,
\[
\ex{|X_t - X_s|^\alpha} = \frac{2^{\alpha/2}}{\sqrt{\pi}}
\Gamma\!\left(\frac{\alpha+1}{2}\right)
\left(\ex{(X_t - X_s)^2}\right)^{\alpha/2} .
\]
The first statement of the corollary can be proved by applying
Kolmogorov continuity theorem for $\alpha > 2/\delta$ and
$\beta = \frac12 \alpha \delta - 1$.
The second statement of the corollary can be proved by applying
Kolmogorov continuity theorem for $\alpha > \frac{2}{\delta - 2\gamma}$
and $\beta = \frac12 \alpha \delta - 1$.
Consider the random event
\begin{align*}
A &= \left\{ \forall \gamma\in(0, \textstyle{\frac12} \delta) : \widetilde X \;
\mbox{satisfies $\gamma$-H\"older condition} \right\}
\\ &=
\left\{
\forall n\in\mathbb{N} : \widetilde X \;
\mbox{satisfies $
\frac12 \left(1-\frac{1}{n}\right) \delta$-H\"older condition} \right\} .
\end{align*}
  (The measurability of $A$ follows from
  the continuity of the process $\widetilde X$).
By the second statement of Corollary~\ref{cor:Gausscontp}  $\pr(A) = 1$.
Thus, $\{\widetilde X_t \ind{A}, \;
t\in[0,t]\}$ is the desired modification which satisfies
H\"older condition for all exponents
$\gamma\in(0, \frac12 \delta)$.

\begin{remark}
\begin{enumerate}
\item
Corollary~\ref{cor:Gausscontp} holds true even without assumption
that the Gaussian process $X$ is centered.

\item
The first statement of Corollary~\ref{cor:Gausscontp}
can be proved with Xavier Fernique's continuity criterion \cite{fernique}
as well.
\end{enumerate}
\end{remark}

\begin{lemma}\label{lem:GContF}
Let $\{X_t, \; t\in [0, T]\}$ be a centered Gaussian process.
Suppose that there exist $\delta > 0$ and a nondecreasing continuous function
$F : [0, T] \to \mathbb{R}$ such that
\begin{equation}\label{eq:lemGContF}
\ex{(X_t - X_s)^2} \le (F(t) - F(s))^\delta
\quad \mbox{for all} \quad
0 \le s \le t \le T.
\end{equation}
Then
\begin{enumerate}
\item The process $X$ have a modification
$\widetilde X$ that has continuous trajectories.
\item If the function $F$
satisfies Lipschitz condition
in an interval $[a,b] \subset [0,T]$,
then for every $\gamma$, $0< \gamma < \frac{1}{2}\delta$,
the process $\widetilde X$ has a modification whose
trajectories satisfy $\gamma$-H\"{o}lder property
on the interval $[a, b]$.
\end{enumerate}
\end{lemma}

\begin{proof}
Without loss of generality, we can assume that the function
$F$ is strictly increasing.
Indeed, if the condition \eqref{eq:lemGContF}
holds true for $F$ being continuous nondecreasing function $F_1$,
it also holds true for $F=F_2$ with $F_2(t)=F_1(t)+t$,
where $F_2$ is a continuous strictly increasing function.

With this additional assumption, the inverse function
$F^{-1}$ is one-to-one, strictly increasing continuous function
$[F(0), F(T)] \to [0, T]$.
Consider a stochastic process
$\{Y_u, \; u \in [F(0), F(T)]\}$,
with
$Y_u = Y_{F^{-1}(u)}$.
The stochastic process $Y$ is centered and Gaussian;
it satisfies condition
\[
\ex{(Y_v - Y_u)^2}
= \ex{(X_{F^{-1}(v)} - X_{F^{-1}(u)})^2}
\le (F(F^{-1}(v)) - F(F^{-1}(u)))^\delta = (v - u)^\delta
\]
for all $F(0) \le u \le v \le F(T)$.
According to Corollary~\ref{cor:Gausscontp},
the process $Y$ has a modification $\widetilde Y$ with continuous trajectories.
Then $\widetilde X$ with $\widetilde X_t = \widetilde Y_{F(t)}$
is a modification of the process $X$ with continuous trajectories.

The second statement of the lemma is a direct consequence
of Corollary~\ref{cor:Gausscontp}. If the function $F$
satisfies Lipschitz condition with constant $L$ on the interval $[a,b]$,
then
\[
\ex{(X_t - X_s)^2} \le L^\delta (t - s)^\delta
\quad \mbox{for all} \quad a \le s \le t \le b ,
\]
which is the main condition for Corollary~\ref{cor:Gausscontp}.
\end{proof}

\subsection{Application of fractional calculus}
The lower and upper Riemann--Liouville fractional integrals of a function
$f\in L^{1}[a,b]$
are defined as follows:
\begin{gather*}
(I_{a+}^\alpha f) (x) = \frac{1}{\Gamma(\alpha)}
\int_a^x \frac{f(t) \, dt}  {(x-t)^{1-\alpha}}, \qquad
(I_{b-}^\alpha f) (x) = \frac{1}{\Gamma(\alpha)}
\int_x^b \frac{f(t) \, dt}  {(t-x)^{1-\alpha}}.
\end{gather*}
The integrals $(I_{a+}^\alpha f) (x)$
and $(I_{b-}^\alpha f) (x)$
are well-defined for almost all $x\in[a,b]$,
and are integrable functions of $x$,
that is
$I_{a+}^\alpha f \in L^1[a,b]$ and
$I_{b-}^\alpha f \in L^1[a,b]$.
Thus, $I_{a+}^\alpha$ and $I_{b-}^\alpha$
might be considered linear operators
$L^1[a,b] \to L^1[a,b]$.

A reflection relation for functions $g(x) = f(a+b-x)$
imply the following relation for their fractional
integrals:
\begin{equation}\label{eq:refl}
(I_{b-}^{\alpha} g)(x) = (I_{a+}^{\alpha} f) (a+b-x);
\end{equation}
see \cite[Chapter~1, Section~2.3]{Samko}.

The integration-by-parts formula is   given, e.g.,
in \cite[Chapter~1, Section~2.3]{Samko}.
\begin{proposition}[integration-by-parts formula]\label{prop:fipi}
Let $\alpha > 0$,
$f \in L^p[a,b]$,
$g \in L^q[a,b]$,
$p \in [1, +\infty]$,
$q \in [1, +\infty]$,
while $\frac{1}{p} + \frac{1}{q} \le 1 + \alpha$
and $\max\Bigl(1 + \alpha - \frac{1}{p} - \frac{1}{q},
\:\allowbreak  \min\!\left(1 - \frac{1}{p}, \:
1 - \frac{1}{q}\right)\Bigr) > 0$.
Then
\[
\int_a^b (I_{a+}^\alpha f)(t) \, g(t) \, dt =
\int_a^b f(t) \, (I_{b-}^\alpha g)(t) \, dt.
\]
\end{proposition}

Now we establish conditions for a function
to be in the range of the fractional
operator $I_{a+}^\alpha$, and we provide
formulas for the preimage, which is called a fractional derivative.
The following statements are the modifications
of the Theorem 2.1 and following corollary in
\cite[Chapter 1]{Samko}.
The formulas for the fractional derivative
are also provided in \cite[Section 2.5]{ManPol2000}.

\begin{theorem}\label{thm:thmeqAbel}
Let $0 < \alpha < 1$.
Consider the integral equation
\begin{equation}\label{eq:AbelEq}
I_{a+}^\alpha f = g
\end{equation}
with unknown function $f \in L^1[a,b]$
and known function (i.e., a parameter) $g \in L^1[a,b]$.
Denote
\[
h(x) = \begin{cases} (I_{a+}^{1-\alpha} g) (x)
& \mbox{if $a < x \le b$}, \\
0
& \mbox{if $x=a$}.
\end{cases}
\]
If $h \in \AC[a,b]$, then equation \eqref{eq:AbelEq}
has a unique (up to equality almost everywhere in $[a,b]$) solution $f$,
namely $f(x) = h'(x)$.
Otherwise, if $h \not\in \AC[a,b]$, then equation \eqref{eq:AbelEq}
has no solutions in $L^1[a,b]$.
If for some $x\in(a,b]$
the integral $(I_{a+}^{1-\alpha} g) (x)$ is not
well-defined,
 then equation \eqref{eq:AbelEq}
does not have solutions in $L^1[a,b]$.
\end{theorem}

\begin{corollary}\label{cor:coreqAbel}
Let $0 < \alpha < 1$.
The integral equation
\eqref{eq:AbelEq} with unknown function
$f \in L^1[a,b]$ and known function $g \in \AC[a,b]$
has a unique solution.
The solution is equal to
\begin{align*}
f(x) &=
(I_{a+}^{1-\alpha}(g'))(x) +
\frac{g(a)}{\Gamma(1-\alpha)\,(x-a)^\alpha}
\\ &=
\frac{1}{\Gamma(1-\alpha)}
\left( \int_a^x \frac{g'(t)\,dt} {(x-t)^\alpha}
+ \frac{g(a)}{(x-a)^\alpha} \right) .
\end{align*}
\end{corollary}
\subsection{Existence of the solution to Volterra integral equation
where the integral operator is an operator
of convolution with integrable singularity
at 0}\label{subsect-4}
Consider Volterra integral equation of the first kind
\begin{equation}\label{eq:Volterra1}
\int_0^x f(t) \, g(x-t) \, dt = y(x), \qquad x \in (0,T],
\end{equation}
with $g(x)$ and $y(x)$ known (parameter) functions
and $f(x)$ unknown function.
Suppose that the function $g(x)$ is integrable
in the interval $(0,T]$ but
behaves asymptotically as a power function in
the neighborhood of $0$:
\[
g(x) \sim  \frac{K}{x^{1-\alpha}}, \qquad x \to 0,
\]
where $0 < \alpha < 1$.
More specifically, assume that
$g(x)$ admits a representation
\begin{equation}\label{eq:reprg}
g(x) = \frac{1}{\Gamma(\alpha) x^{1-\alpha}}
+ (I^{\alpha}_{0+} h)(x) =
\frac{1}{\Gamma(\alpha)}
\left( \frac{1}{x^{1-\alpha}} + \int_0^x \frac{h(t)\, dt}
{(x-t)^{1-\alpha}} \right),
\end{equation}
where
$\Gamma(\alpha)$ is a gamma function,
$I^{\alpha}_{0+}h$ is a lower Riemann--Liouville
fractional integral of $h$,
\[
(I^{\alpha}_{0+} h)(x) =
\frac{1}{\Gamma(\alpha)}
\int_0^x \frac{h(t)\, dt} {(x-t)^{1-\alpha}}
,
\]
and $h(x)$ is a absolutely continuous function.

The sufficient conditions for existence
and uniqueness of the solution to integral equation
claimed in \cite[Section 2.1-2]{ManPol2000}
are not satisfied.
The kernel of the integration operator in \eqref{eq:Volterra1}
is unbounded, and
$y(0)$ might be nonzero.

But we use Remark~2 in
\cite[Section 2.1-2]{ManPol2000}.
We reduce the Volterra integral equation of the
first kind to a Volterra integral equation
of the second kind
similarly as it is done for
regular functions $g(x)$;
compare with ~\cite[Section 2.3]{ManPol2000}
for the case of regular $g(x)$.

For the next theorem we keep in mind that if
a function $f$ is a solution to \eqref{eq:Volterra1},
then every function that is equal to $f$
almost everywhere on $[0,T]$
is also a solution to \eqref{eq:Volterra1}.
\begin{theorem}\label{thm:EUVolterra1}
Let $y,\,h \in C^1[0,T]$ and $g$ be defined in \eqref{eq:reprg}.
Then the equation \eqref{eq:Volterra1} has
a unique (up to equality almost everywhere)
solution $f \in L^1[0,T]$.
The solution is (more precisely, some of almost-everywhere equal solutions are)
continuous in the left-open interval $(0,T]$.
\end{theorem}

\begin{proof}
Substitute \eqref{eq:reprg} into \eqref{eq:Volterra1}:
\begin{gather*}
\int_0^x f(t)
\left( \frac{1}{\Gamma(\alpha) (x-t)^{1-\alpha}}
+ (I^{\alpha}_{0+} h)(x-t) \right) dt
= y(x), \\
(I^{\alpha}_{0+} f)(x) +
\int_0^x f(t) \, (I^{\alpha}_{0+} h)(x-t) \, dt
= y(x) .
\end{gather*}
Denote $h_x(t) = h(x-t)$.
According to equation \eqref{eq:refl},
 the fractional integrals of $h$ and $h_x$
satisfy the relation
$(I^{\alpha}_{0+} h)(x-t)
= (I^{\alpha}_{x-} h_x)(t)$.
Hence, equation \eqref{eq:Volterra1}
is equivalent to the following one:
\begin{equation}\label{eq:Volterra1-106}
(I^{\alpha}_{0+} f)(x) +
\int_0^x f(t) \, (I^{\alpha}_{x-} h_x)(t) \, dt
= y(x).
\end{equation}

Now apply the integration-by-parts formula. We have $f \in
L^1[0,x]$, $h_x \in L^\infty[0,x]$, and $1 + 0 < 1 + \alpha$.
Hence, by Proposition~\ref{prop:fipi},
\[
\int_0^x f(t) \, (I^{\alpha}_{x-} h_x)(t) \, dt =
\int_0^x (I^{\alpha}_{0+} f)(t) \, h_x(t) \, dt.
\]

It means that   equation \eqref{eq:Volterra1-106} is equivalent to the following ones:
\begin{gather*}
(I^{\alpha}_{0+} f)(x) +
\int_0^x (I^{\alpha}_{0+} f)(t) \,  h_x(t) \, dt
= y(x), \end{gather*}
and \begin{gather*}
(I^{\alpha}_{0+} f)(x) +
\int_0^x (I^{\alpha}_{0+} f)(t) \,  h(x-t) \, dt
= y(x) .
\end{gather*}

Denote
$F = I^{\alpha}_{0+} f$,
and obtain a Volterra integral equation of the second kind:
\begin{equation}\label{eq:Volterra2}
F(x) = y(x) - \int_0^x F(t) \,  h(x-t) \, dt.
\end{equation}
Equation \eqref{eq:Volterra2}
has a unique solution in $C[0,T]$,
as well as in $L^1[0,T]$.
In   other words,
\eqref{eq:Volterra2} has a unique
integrable solution, and this solution
is a continuous function.

According to Theorem~\ref{thm:thmeqAbel},
either unique (up to almost-everywhere equality) function $f$, or no functions $f$ correspond to
the function $F$.
Thus, all integrable solution to integral equation \eqref{eq:Volterra1}
are equal almost everywhere.

Now we construct a solution to equation
\eqref{eq:Volterra1} that is continuous and
integrable on $(0,T]$.
Differentiating \eqref{eq:Volterra2}, we obtain
\[
F'(x) = y'(x) - F(x)\, h(0) - \int_0^x F(t) \, h'(x-t) \, dt,
\]
whence $F \in C^1[0,T]$.
According to Corollary~\ref{cor:coreqAbel},
the integral equation $F = I_{0+}^\alpha f$
has a unique solution $f\in L^1[0,T]$,
which is equal to

\begin{equation}\label{eq:freprofF}
f(x)
= \frac{1}{\Gamma(1-\alpha)}
\left( \int_0^x \frac{F'(t)\,dt}{(x-t)^\alpha} +
\frac{F(0)}{x^\alpha} \right) .
\end{equation}
The constructed function $f(x)$
is continuous and integrable in $(0,T]$,
and $f(x)$ is a solution to \eqref{eq:Volterra1}.
\end{proof}

\begin{remark}
In Theorem~\ref{thm:EUVolterra1}
the condition $h \in C^1[0,T]$
can be relaxed and replaced with
the condition $h \in \AC[0,T]$.
In other words, if the function $h$
is absolutely continuous
but is not continuously differentiable,
the statement of Theorem~\ref{thm:EUVolterra1}
still holds true.
\end{remark}

\subsubsection
[Example: $g(x) = \exp(\beta x) x^{\alpha-1} / \Gamma(\alpha)$
and $y(x) = 1$]
{Example: \boldmath$g(x) = \exp(\beta x) x^{\alpha-1} / \Gamma(\alpha)$
and $y(x) = 1$}\label{sect:Example}
It is well known that
\begin{equation}\label{eq:con1}
\int_0^x \frac{1}{\Gamma(1-\alpha) t^\alpha} \,
\frac{1}{\Gamma(\alpha) (x-t)^{1-\alpha}} \, dt = 1 .
\end{equation}

In this section, we prove that the equation
\begin{equation}\label{eq:con2}
\int_0^x f(t) \,
\frac{e^{(x-t)\beta }}{\Gamma(\alpha) (x-t)^{1-\alpha}} \, dt = 1
\end{equation}
has an integrable solution.
According to \eqref{eq:con1},
$f(x) = x^{-\alpha} / \, \Gamma(1-\alpha)$
is a solution to \eqref{eq:con2} if $\beta = 0$.

Denote
\begin{equation}\label{eq:concreteg}
g(x) = \frac{\exp(\beta x)} {\Gamma(\alpha) x^{1-\alpha}} .
\end{equation}
Demonstrate that $g(x)$ admits a representation \eqref{eq:reprg}.
To construct $h$, we need
Kummer confluent hypergeometric function
\cite{Wolfram1F1}:
\[
\hyperIFI(a;b;z) = \frac{1}{\mathrm{B}(a,\: b-a)}
\int_0^1 e^{zt} t^{a-1} (1-t)^{b-a-1} \, dt,
\qquad
0<a<b, \quad z\in\mathbb{C} .
\]
For $a$ and $b$ fixed,
$\hyperIFI(a;b;\,\cdot\,)$ is an entire function.
Its derivative equals
\[
\frac{\partial}{\partial z}
\hyperIFI(a;b;z) = \frac{a}{b}
\hyperIFI(a+1;\:b+1;\:z) .
\]
For all $0<a<b$ and $z\in\mathbb{R}$
\[
\hyperIFI(a; b; z) > 0, \qquad
\hyperIFI(a; b; 0) = 1.
\]
Notice that if $0<\alpha<1$ and $x>0$, then
\[
\frac{1}{\mathrm{B}(\alpha, \: 1-\alpha)}
\int_0^x \frac{\exp(zt) \, dt}
{t^{1-\alpha} (x-t)^\alpha} =
\hyperIFI(\alpha; 1; xz) .
\]

Being considered an equation for unknown $h$,
\eqref{eq:reprg} is equivalent to $I_{0+}^\alpha h = g_0$,
where
\[
g_0(x) = g(x) - \frac{1}{\Gamma(\alpha) x^{1-\alpha}}
= \frac{e^{\beta x} - 1}{\Gamma(\alpha) x^{1-\alpha}} .
\]
Then
\begin{align*}
(I_{0+}^{1-\alpha} g_0)(x) =
\frac{1}{\mathrm{B}(\alpha, \: 1-\alpha)}
\int_0^x \frac{e^{\beta t} - 1}{t^{1-\alpha}
(1-t)^\alpha} \, dt =
\hyperIFI(\alpha; 1; \beta x) - 1 .
\end{align*}
Besides, $\hyperIFI(\alpha; 1; \beta x) - 1$ is an absolutely
continuous function in $x$, and\linebreak  $\hyperIFI(\alpha; 1;
\beta x) - 1 = 0$ if $x=0$. According to Theorem~\ref{thm:thmeqAbel}, the
equation $I_{0+}^\alpha h = g_0$ has the unique solution $h =
L^1[0,T]$, which is equal to
\begin{equation}\label{eq:concreteTTg}
h(x) = \frac{\partial(\hyperIFI(\alpha; 1; \beta x) - 1)}{\partial x} =
\alpha \beta \hyperIFI(\alpha+1;\: 2;\: \beta x).
\end{equation}
The constructed function $h(x)$ is a solution to
\eqref{eq:reprg} and is continuously differentiable.

In summary, $h \in C^1[0,T]$,
$y(x) = 1$, and $y \in C^1[0,T]$.
According to Theorem~\ref{thm:EUVolterra1}
the integral equation
\begin{equation}\label{eq:Volterra1y1}
\int_0^x f(t) \, g(x-t) \, dt = 1, \qquad x\in(0,T],
\end{equation}
has a unique solution $f \in L^1[0,T]$
(up to equality almost everywhere).
The solution is continuous in $(0,T]$.
\begin{remark} The fact that the functions
$g$ and $h$ defined in \eqref{eq:concreteg} and
\eqref{eq:concreteTTg}, respectively,
satisfy \eqref{eq:reprg}, can be checked directly.
For such verification, one can apply
Lemma~2.2(i) from \cite{Norros}.
\end{remark}

\subsubsection{Positive solution to the Volterra
integral equation}
\begin{theorem}\label{thm:positive}
Let the conditions of Theorem~\ref{thm:EUVolterra1}
hold true.
Additionally, let
\[
y(x)>0, \quad y'(x)\ge 0, \quad
h(x) < 0 \qquad
\mbox{for all $x\in[0,T]$}.
\]
Then the continuous solution $f(x)$ to \eqref{eq:Volterra1}
attains only positive values in $(0,T]$.
\end{theorem}

\begin{proof}
Notice that \eqref{eq:Volterra2} implies $F(0) = y(0) > 0$.
Taking this into account, let's  differentiate both sides of \eqref{eq:Volterra2}
the other way:
\begin{gather}
F(x) = y(x) - \int_0^x F(x-t) \,  h(t) \, dt,
\nonumber \\
F'(x) = y'(x) - F(0) \, h(x)
- \int_0^x F'(x-t) \,  h(t) \, dt
\label{eq:Volterra2dF} .
\end{gather}

Let us prove that $F'(x)>0$ in $[0,T]$
by contradiction.
Assume the contrary, that is $\exists x \in [0,1] : F'(x) \le 0$.
Since the function $F'(x)$ is continuous in $[0,T]$,
the contrary implies the existence of the minimum in
\[
x_0 = \min\{x\in[0,T] : F'(x) \le 0\} .
\]
But for $x=x_0$ the left-hand side in
\eqref{eq:Volterra2dF} is less or equal then zero,
while the right-hand side is greater than zero.
Thus, \eqref{eq:Volterra2dF} does not hold true.

The proof also works for $x_0 = 0$.
There is a contradiction. Thus, we have proved that $F'(x)>0$ for all $x\in[0,T]$.
By \eqref{eq:freprofF},
$f(x)>0$ for all $x\in(0,T]$.
\end{proof}

\end{document}